\documentclass{proc-l}
\usepackage{tikz}
\usetikzlibrary{positioning,automata,calc,arrows}
\usepackage{tikz-cd}
\newtheorem{theorem}{Theorem}[section]
\newtheorem{lemma}[theorem]{Lemma}
\theoremstyle{definition}
\newtheorem{definition}[theorem]{Definition}
\newtheorem{example}[theorem]{Example}

\theoremstyle{remark}
\newtheorem{remark}[theorem]{Remark}
\newtheorem{corollary}[theorem]{Corollary}
\numberwithin{equation}{section}
\begin{document}
%---------------------------------------------%
% Heading Entry %-----------------------------%
%---------------------------------------------%
    \title{When do the Kahn-Kalai Bounds provide nontrivial information?}
    
    \author{Bryce Alan Christopherson}
    \address{Department of Mathematics, University of North Dakota, Grand Forks, ND 58202}
    \email{bryce.christopherson@UND.edu}
    
    \author{Jack Baretz}
    \address{Department of Mathematics, University of North Dakota, Grand Forks, ND 58202}
    \email{jack.baretz@UND.edu}
    
    %    General info

    \thanks{\phantom{a}\\Conflict of Interest Statement: The authors declare none. \\Funding Statement: This work received no specific grant from any funding agency, commercial or not-for-profit sectors.  \\Data Availability Statement: Data sharing not applicable to this article as no datasets were generated or analysed during the current study.}
    
    
%---------------------------------------------%
% Abstract %----------------------------------%
%---------------------------------------------%
    \begin{abstract}
        The Park-Pham theorem (previously known as the Kahn-Kalai conjecture), bounds the critical probability, $p_c(\mathcal{F})$, of the a non-trivial property $\mathcal{F}\subseteq 2^X$ that is closed under supersets by the product of a universal constant $K$, the expectation threshold of the property, $q(\mathcal{F})$, and the logarithm of the size of the property's largest minimal element, $\log\ell(\mathcal{F})$.  That is, the Park-Pham theorem asserts that $p_c(\mathcal{F})\leq Kq(\mathcal{F})\log\ell(\mathcal{F})$.  Since the critical probability $p_c(\mathcal{F})$ always satisfies $p_c(\mathcal{F})<1$, one may ask when the upper bound posed by Kahn and Kalai gives us more information than this--that is, when is it true that $Kq(\mathcal{F})\log\ell(\mathcal{F}) < 1$?  In this short note, we provide a number of necessary conditions for this to happen and give a few sufficient conditions for the bounds to provide new (and, in fact, asymptotically perfect) information along the way.  In the most interesting case where $\ell(\mathcal{F}_n)\rightarrow \infty$, we prove the following relatively strong necessary condition for the Kahn-Kalai bounds to provide nontrivial information: For every positive integer $t$, every collection of all-but-$t$ of the minimal elements of $\mathcal{F}_n$ may have nonempty intersection for only finitely many $n$.  Consequently, not only must the number of minimal elements become arbitrarily large, but so too must the size of any cover.  Intuitively, this means that such sequences $\mathcal{F}_n$ must occupy an ever-widening `wedge' in $2^{X_n}$:  the further $\mathcal{F}_n$ climbs up $2^{X_n}$ in one area, the further it must spread down and across $2^{X_n}$ in another. 
    \end{abstract}

    \maketitle

%---------------------------------------------%
% Body %--------------------------------------%
%---------------------------------------------%
\section{Introduction}
%---------------------------------------------%
    We will say that the Kahn-Kalai bounds \textit{provide nontrivial information} for the sequence of nontrivial upper sets $\mathcal{F}_n \subseteq 2^{X_n}$ if there exists $N$ such that for all $n\geq N$, we have $Kq(\mathcal{F}_n)\log\ell(\mathcal{F}_n) < 1$.  While nice to know in their own right (if only for the obvious reason that the constant function $1$ is computationally less expensive to calculate or approximate than $Kq(\mathcal{F}_n)\log\ell(\mathcal{F}_n)$ is), both items are also necessary in extensions of the Kahn-Kalai bounds to $p$-biased measures on upper sets in generic partially ordered sets \cite{christopherson2024conditional} and allow one to check whether certain stronger versions of the bound can be used instead \cite{przybylowski2023thresholds}.

    To introduce the theme of the problem a bit more clearly, here are some simple observations as to how this works. 
    %---------------------------------------------%
    \begin{remark}
        For the Kahn-Kalai bounds to provide nontrivial information, it is not enough that $\ell(\mathcal{F}_n)$ is constant or slow growing if $q(\mathcal{F}_n)$ is too large and does not decrease quickly enough.  Likewise, it is not enough that $q(\mathcal{F}_n) = o(1)$ if $\ell(\mathcal{F}_n)$ grows quickly.  More intuitively, the speed at which the upper sets $\mathcal{F}_n$ may 'climb up' $2^{X_n}$ (as measured by the size of the largest minimal element of $\mathcal{F}_n$) determines (or is determined by) how quickly $q(\mathcal{F}_n)$ decreases.  
    \end{remark}
    %---------------------------------------------%
    \begin{remark}
        If both $q(\mathcal{F}_n) = o(1)$ and $\ell(\mathcal{F}_n)$ is bounded, then the Kahn-Kalai bounds provide nontrivial information.  While this covers many interesting cases, there are many other interesting cases where this does not happen. For instance, even in the relatively tame classical setting of random graphs, $\ell(\mathcal{F}_n)$ is not bounded if $\mathcal{F}_n$ corresponds to the subsets of the edges of $K_n$ such that the corresponding edge-induced subgraphs are connected, or have a Hamiltonian cycle, or are non-planar, or  contain a copy of the star graph $S_{n-m}$ for some fixed $m > 0$, or contain a copy of $G_n$ for any sequence of graphs with $|V(G_n)|\leq n$ and $|E(G_n)|<|E(G_{n+1})|$, etc.
    \end{remark}
%---------------------------------------------%
\section{Background}\label{background}
%---------------------------------------------%
    Given a finite set $X$, we say that a \textit{property} of $X$ is a subset $\mathcal{F}\subseteq 2^X$.  A property of $X$ is said to be an \textit{upper set} if $B \in \mathcal{F}$ whenever $A \subseteq B$ for some $A \in \mathcal{F}$.  For a customary example, the subset of subgraphs $H$ of a given graph $G$ such that $H$ contains a subgraph isomorphic to some target graph $K$ forms an upper set, as the addition of vertices or edges to a graph preserves subgraph inclusion \cite{erdds1959random}.
    
    We will regard \textit{thresholds} for \textit{sequences} of upper sets in $2^X$ particularly.  For $p \in [0,1]$ define the product measure $\mu_{p}$ on $X$ by $\mu_{p}(S)=p^{|S|}(1-p)^{|X|-|S|}$ and extend this to a probability measure on $2^X$ by $\mu_p(\mathcal{F}):=\sum_{S \in \mathcal{F}}\mu_p(S)$.
    %---------------------------------------------%
    \begin{definition}[Threshold]\label{powerset threshold definition}
        Given a sequence of finite sets $X_n$ and a sequence of upper sets $\mathcal{F}_n \subseteq 2^{X_n}$, we say that a function $p^*(n)$ is a \textit{threshold} for $\mathcal{F}_n$ if $\mu_{p(n)}\big(\mathcal{F}_n\big) \rightarrow 0$ when $\frac{p(n)}{p^*(n)} \rightarrow 0$ and $\mu_{p(n)}\big(\mathcal{F}_n\big) \rightarrow1$ when $\frac{p(n)}{p^*(n)} \rightarrow \infty$.
    \end{definition}
    %---------------------------------------------%
    Every sequence of nontrivial upper sets $\mathcal{F}_n$ has a threshold.  One choice in particular is the \textit{critical probability of $\mathcal{F}$}, denoted $p_c(\mathcal{F})$, which is the unique value $p \in [0,1]$ such that $\mu_p(\mathcal{F})=\frac{1}{2}$.  That is, $p^*(n) := p_c\big(\mathcal{F}_n\big)$ is a threshold for the sequence $\mathcal{F}_n$ \cite{janson2011random}.

    The Kahn-Kalai conjecture \cite{kahn2007thresholds} (now, the Park-Pham theorem \cite{park2024proof}) concerns the relationship between the threshold obtained from the critical probability $p_c$ and the \textit{expectation threshold} of an upper set, a quantity derived from considering certain covers of upper sets. Given an upper set $\mathcal{F} \subseteq 2^X$, we say that a subset $\mathcal{G}\subseteq 2^X$ is a \textit{cover} of $\mathcal{F}$ if $\mathcal{F} \subseteq \bigcup_{S \in \mathcal{G}} \langle S \rangle$, where $\langle S \rangle$ denotes the upper set generated by $S$, i.e. $\langle S \rangle =\left\{T: S \subseteq T\right\}$.  For $p \in [0,1]$, we say that $\mathcal{F}$ is \textit{$p$-small} if there exists a cover $\mathcal{G}$ of $\mathcal{F}$ such that $\sum_{S \in \mathcal{G}}p^{|S|} \leq \frac{1}{2}$.  The \textit{expectation threshold} of an upper set $\mathcal{F}$, denoted $q(\mathcal{F})$, is then defined to be the largest $p$ for which $\mathcal{F}$ is $p$-small.
    
    Now, let $\mathcal{F}_0$ denote the set of minimal elements of $\mathcal{F}$ let $\ell_0(\mathcal{F})=\textrm{max}\left\{|S|: S \in \mathcal{F}_0\right\}$ and write $\ell(\mathcal{F})=\textrm{max}\left\{\ell_0(\mathcal{F}),2\right\}$. The Park-Pham theorem \cite{park2024proof}(previously the Kahn-Kalai conjecture \cite{kahn2007thresholds}) bounds the critical probability by a logarithmic factor of the expectation threshold.
    %---------------------------------------------%
    \begin{theorem}[Park-Pham Theorem]\label{parkpham}
        There exists a universal constant $K$ such that for every finite set $X$ and every nontrivial upper set $\mathcal{F} \subseteq 2^X$, 
        $$
            q(\mathcal{F}) \leq p_c(\mathcal{F}) \leq Kq(\mathcal{F})\log \ell(\mathcal{F}).
        $$
    \end{theorem}
    %---------------------------------------------%
    A version of the above, given by Bell \cite{bell2023park}, finds an improved value of $K$ and shows that for every finite set $X$ and every nontrivial upper set $\mathcal{F} \subseteq 2^X$, we get $q(\mathcal{F}) \leq p_c(\mathcal{F}) \leq 8q(\mathcal{F})\log \left(2\ell_0(\mathcal{F})\right)$.  It is worth noting that one can optimize $K$ further.  For instance, it was shown in \cite{vu2023short} that one can set $K \approx 3.998$ in the case where $\ell(F_n) \rightarrow \infty$.  Without this restriction, Park and Vondrak \cite{park2024simple} showed $p_c \leq 4.5q \log\left(2\ell_0(\mathcal{F})\right)$ in all cases, which is the best value at present.  It is interesting to note as well that the analog of Theorem~\ref{parkpham} for lower sets does not end up working out as one might expect \cite{gunby2023down,warnke2024note}.

    We make one slight deviation in our notation throughout, to provide more detail.  Here, we will let $K$ will refer to any universal constant satisfying Theorem~\ref{parkpham} for a specified class of sequences $(\mathcal{F}_n)$ of nontrivial monotonically increasing properties, though the motivation is that one should think of it as the best possible $K$ for whatever class of properties to which $(\mathcal{F}_n)$ belong.  For instance, $K\approx 3.998$ is the best known $K$ for the class of sequences of nontrivial upper sets $(\mathcal{F}_n)$ with $\ell(\mathcal{F}_n)\rightarrow \infty$ \cite{vu2023short}, while $K=4.5$ is the best known $K$ for the unrestricted case \cite{park2024simple}.
%---------------------------------------------%
\section{Easy Positive Cases}
%---------------------------------------------%
%---------------------------------------------%
    Since $q(\mathcal{F}_n) \leq p_c(\mathcal{F}_n)\leq Kq(\mathcal{F}_n)\log \ell(\mathcal{F}_n)$, we know that there is an interval of width 
    $$
        Kq(\mathcal{F}_n)\log \ell(\mathcal{F}_n) - q(\mathcal{F}_n) = (K\log\ell(\mathcal{F}_n) - 1)q(\mathcal{F}_n)
    $$
    in which $p_c(\mathcal{F}_n)$ resides.  When this width shrinks to zero (that is, when the Kahn-Kalai bounds provide asymptotically perfect information), we certainly get new information.  This is the best situation and the easiest one to deal with.

    A simple observation:  If $\lim\limits_{n\rightarrow \infty}(K\log\ell(\mathcal{F}_n) - 1)q(\mathcal{F}_n) = 0$, then 
    $$
        \lim\limits_{n\rightarrow \infty}\frac{(K\log\ell(\mathcal{F}_n) - 1)}{1/q(\mathcal{F}_n)} = 0.
    $$
    So, the Kahn-Kalai bounds provide asymptotically perfect information if and only if $K\log\ell(\mathcal{F}_n) - 1 \ll \frac{1}{q(\mathcal{F}_n)}$.  It is not hard to see that we can do away with the constant term and factor of $K$.  This is reasonably obvious: Since $K\log\ell(\mathcal{F}_n) \geq K$, it is straightforward to see that $(K\log\ell(\mathcal{F}_n)-1) \geq K-1$ and $Kq(\mathcal{F}_n)\log\ell(\mathcal{F}_n)-q(\mathcal{F}_n) \geq q(\mathcal{F}_n)(K-1)$.  So, if the Kahn-Kalai bounds provide perfect information, then $q(\mathcal{F}_n)\rightarrow 0$ and, necessarily, so too must $p_c(\mathcal{F}_n)$ and $q(\mathcal{F}_n)\log\ell(\mathcal{F}_n)$.
    \begin{theorem}\label{perf info theorem}
        The Kahn-Kalai bounds provide asymptotically perfect information if and only if $\log \ell(\mathcal{F}_n) \ll \frac{1}{q(\mathcal{F}_n)}$.
    \end{theorem}
    %---------------------------------------------%
    So, a straightforward answer.
    %---------------------------------------------%
    \begin{remark}\label{interesting case remark}
        In the case where $\ell(\mathcal{F}_n)\rightarrow \infty$, suppose the Kahn-Kalai bounds provide nontrivial information.  Immediately, this yields the existence of $N$ such that $\log\ell(\mathcal{F}_n)<\frac{1}{Kq(\mathcal{F}_n)}$ for all $n \geq N$ and $\frac{1}{q(\mathcal{F}_n)}\rightarrow \infty$.  That is, if $\ell(\mathcal{F}_n)\rightarrow \infty$, then getting any information at all automatically yields $\frac{1}{q(\mathcal{F}_n)}\rightarrow \infty$, so the only barrier to obtaining perfect information is that the latter may not grow at a fast enough rate relative to the former.
    \end{remark}
    %---------------------------------------------%
    \begin{remark}
        If there is $0<C \leq 1$ such that $\lim\limits_{n\rightarrow \infty}Kq(\mathcal{F}_n)\log\ell(\mathcal{F}_n)=C$, then the Kahn-Kalai bounds provide new--but not perfect--information and $q(\mathcal{F}_n)\log\ell(\mathcal{F}_n) \rightarrow \frac{C}{K} \leq \frac{1}{K}$.  That is, asymptotically perfect information is equivalent to having $\frac{1}{q(\mathcal{F}_n)}$ grow much faster than $\log\ell(\mathcal{F}_n)$ while new information (supposing the limit of the Kahn-Kalai bound exists) is equivalent to having $\frac{1}{q(\mathcal{F}_n)}$ grow, asymptotically, at least $K$ times faster than $\log\ell(\mathcal{F}_n)$.  
        
        If limit of the Kahn-Kalai bound \textit{does not} exist, then we are in a messier situation and should not expect a nice if-and-only-if statement that completely characterizes new information in terms of the growth rates of the quantities involved in the bounds.  However, as we will do in the next sections, we can still obtain a decent partial characterization via some reasonably strong necessary conditions.
    \end{remark}
%---------------------------------------------%
%---------------------------------------------%
\section{A Simple Case Where We Do Not Get New Information}
%---------------------------------------------%
%---------------------------------------------%
    There are some cases where the Kahn-Kalai bounds do not provide nontrivial information that may not be immediately obvious.  The easiest is for sequences $\mathcal{F}_n$ of nontrivial (i.e.  $\mathcal{F}_n\neq \emptyset, 2^{X_n}$)  principal upper sets (i.e. $\mathcal{F}_n$ is generated by a single element).
    %---------------------------------------------%
    \begin{theorem}\label{expectation threshold for prinicpal upper set}
        Suppose $\mathcal{F}_n$ is a sequence of nontrivial principal upper sets.  Then, the Kahn-Kalai bounds do not provide nontrivial information if $K \geq 2$.
    \end{theorem}
    %---------------------------------------------%
    \begin{proof}
        Suppose $\mathcal{F}_n$ is a nontrivial principal upper set in $2^{X_n}$, i.e. $(\mathcal{F}_n)_0=\left\{S_n \right\}$ for some element $S_n \in 2^{X_n}$ with $S_n \neq \emptyset$ and $S_n\neq X_n$ (where, recall, $(\mathcal{F}_n)_0$ denotes the set of minimal elements of $\mathcal{F}_n$). Since $(\mathcal{F}_n)_0$ is a cover of $\mathcal{F}_n$, we have
        $q(\mathcal{F}_n) \geq \textrm{max}\left\{p \in [0,1] : \sum_{S \in (\mathcal{F}_n)_0}p^{|S|} \leq 1/2\right\}$.  But, if $(\mathcal{F}_n)_0=\left\{S_n\right\}$, then the largest value of $p$ is such that $p^{|S_n|} \leq \frac{1}{2}$ is simply $p=2^{-1/|S_n|}$.  So, $q(\mathcal{F}_n) \geq 2^{-1/|S_n|}$ and $Kq(\mathcal{F}_n)\log\ell(\mathcal{F}_n)\geq K2^{-1/|S_n|} \geq K/2 \geq 1$.
    \end{proof}
    %---------------------------------------------%
    A tiny variation on the argument above shows that we can also rule out any non-principal upper set that is covered by a principal upper set.
    %---------------------------------------------%
    \begin{lemma}\label{intersection no-go lemma}
        Let $\mathcal{F}_n$ be a sequence of nontrivial upper sets whose minimal elements do not eventually have empty intersection--i.e., there does not exist some $N$ such that $\cap_{S \in (\mathcal{F}_n)_0}S = \emptyset$ for all $n > N$.  Then, the Kahn-Kalai bounds do not provide nontrivial information if $K \geq 2$.
    \end{lemma}
    %---------------------------------------------%
    \begin{proof}
        If $\cap_{S \in (\mathcal{F}_n)_0}S \neq \emptyset$, then observe that $\mathcal{G}=\left\{\cap_{S \in (\mathcal{F}_n)_0}S \right\}$ is a cover of $\mathcal{F}_n$.  So, $q(\mathcal{F}_n) \geq 2^{\frac{-1}{|\cap_{S \in (\mathcal{F}_n)_0}S|}}$ and $Kq(\mathcal{F}_n)\log\ell(\mathcal{F}_n)\geq K2^{-\frac{1}{|\cap_{S \in (\mathcal{F}_n)_0}S|}} \geq K/2 \geq 1$. 
    \end{proof}
    %---------------------------------------------%
    Via the contrapositive, $Kq(\mathcal{F}_n)\log\ell(\mathcal{F}_n) < 1$ for all $n > N$ for some $N$ implies that $\cap_{S \in (\mathcal{F}_n)_0}S = \emptyset$ must also hold.  That is, the minimal elements of $\mathcal{F}_n$ must eventually `spread out' a little bit. In the following section, we will generalize this and show that if $\ell(\mathcal{F}_n)\rightarrow \infty$, then they must spread out \textit{quite} a bit.
%---------------------------------------------%
%---------------------------------------------%
\section{A Less Simple Case Where We Do Not Get New Information}
%---------------------------------------------%
%---------------------------------------------%
    Notice that if $\cap_{S \in (\mathcal{F}_n)_0}S = \emptyset$, then each nontrivial cover of $\mathcal{F}_n$ must have at least two elements.  Likewise, if the intersections of all-but-one of the elements of $(\mathcal{F}_n)_0$ are empty, then any nontrivial cover of $\mathcal{F}_n$ must have at least three elements.  Compactly, this means that $\cup_{S \in (\mathcal{F}_n)_0}\left(\cap_{S\neq H \in (\mathcal{F}_n)_0}H\right)=\emptyset$.  We will use this line of reasoning to determine some necessary conditions for the Kahn-Kalai bounds to provide nontrivial information.  First, we give the following definition for notational convenience:
    % %---------------------------------------------%
    \begin{definition}
        The \textit{covering dimension} of an upper set $\mathcal{F} \subseteq 2^X$, denoted $\textrm{dim}(\mathcal{F})$, is the minimum number of elements in any nontrivial cover $\mathcal{G}$ of $\mathcal{F}$.  That is, $$\textrm{dim}(\mathcal{F})=\textrm{min}\left\{ |\mathcal{G}| : \mathcal{G} \subseteq 2^X, \enskip \mathcal{F} \subseteq \langle \mathcal{G} \rangle \neq 2^X  \right\}.$$
    \end{definition}
    %---------------------------------------------%
    For an example, any principal upper set or upper set covered by a principal upper set (i.e. an upper set whose minimal elements intersect nontrivially) has covering dimension 1.  This was the simple case addressed in the previous section and, analogously to it, the covering dimension of an upper set $\mathcal{F}$ will let us estimate $q(\mathcal{F})$.
    %---------------------------------------------%
    \begin{lemma}\label{q estimate}
        $(2\textrm{dim}\mathcal{F})^{-1} \leq q(\mathcal{F}) \leq (2\textrm{dim}\mathcal{F})^{-1/\ell(\mathcal{F})}$
    \end{lemma}
    %---------------------------------------------%
    \begin{proof}
        Let $\mathcal{G}$ be a cover of $\mathcal{F}$ such that $q(\mathcal{F}) = \textrm{max}\left\{p \in [0,1] : \sum_{S \in \mathcal{G}}p^{|S|} \leq 1/2\right\}$. Taking any $S_1,\hdots,S_{\textrm{dim}\mathcal{F}} \in \mathcal{G}$, we have $q(\mathcal{F})^{|S_1|}+\hdots+q(\mathcal{F})^{|S_{\textrm{dim}\mathcal{F}}|} \leq 1/2$.  Noting $\textrm{max}\left\{|S| : S \in \mathcal{G}\right\} \leq \ell(\mathcal{F})$, then $\textrm{dim}(\mathcal{F})q(\mathcal{F})^{\ell(\mathcal{F})}\leq q(\mathcal{F})^{|S_1|}+\hdots+q(\mathcal{F})^{|S_{\textrm{dim}\mathcal{F}}|} \leq 1/2
        $ and $q(\mathcal{F}) \leq \big(2\textrm{dim}(\mathcal{F})\big)^{-1/\ell(\mathcal{F}_n)}$.

        For the other inequality, let $\mathcal{H}$ be any cover of $\mathcal{F}$ with $\textrm{dim}(\mathcal{F})=|\mathcal{H}|$ and let $p^*$ be the unique $p$ such that $\sum_{s \in \mathcal{H}}(p^*)^{|S|}=1/2$.  Then, $p^* \leq q(\mathcal{F})$. Similarly, notice $1/2=\sum_{s \in \mathcal{H}}(p^*)^{|S|}\leq \textrm{dim}(\mathcal{F})(p^*)$, so $p^*\geq (2\textrm{dim}\mathcal{F})^{-1}$ and $q(\mathcal{F}) \geq (2\textrm{dim}\mathcal{F})^{-1}$.
    \end{proof}
    %---------------------------------------------%
    \begin{remark}
        An obvious fact this tells us is that, to get new information, the dimension need only be large enough relative to a function of $\ell(\mathcal{F}_n)$ (i.e., if $\textrm{dim}(\mathcal{F}_n) > \frac{1}{2}\big(K\log\ell(\mathcal{F}_n)\big)^{\ell(\mathcal{F}_n)}$ for all sufficiently large $n$, the Kahn-Kalai bounds provide nontrivial information).  Likewise, combining the observation in Corollary~\ref{perf info theorem} with Lemma~\ref{q estimate}, it is easy to see that the Kahn-Kalai bounds provide asymptotically perfect information whenever $\log\ell(\mathcal{F}_n) \ll (2\textrm{dim}\mathcal{F}_n)^{1/\ell(\mathcal{F}_n)}$.  However, it is exceedingly rare that this is of any practical use, since it requires both that $\textrm{dim}\mathcal{F}_n$ is known and quite large.
    \end{remark}
    %---------------------------------------------%
    \begin{example}
        Let $\mathcal{F}_n$ corresponds to the subsets of the edges of $K_n$ such that the corresponding edge-induced subgraphs are connected.  Then, $\ell(\mathcal{F}_n)=n-1$ and the dimension of $\mathcal{F}_n$ is the number of spanning trees in $K_n$, so $\textrm{dim}(\mathcal{F}_n)=n^{n-2}$.  Hence, $$\lim\limits_{n\rightarrow \infty} \frac{\log\ell(\mathcal{F}_n)}{ (2\textrm{dim}\mathcal{F}_n)^{1/\ell(\mathcal{F}_n)}}=\lim\limits_{n\rightarrow \infty} \frac{\log(n-1)}{ (2n^{n-2})^{1/(n-1)}}=0$$ and the Kahn-Kalai bounds provide perfect information.
    \end{example}
    %---------------------------------------------%
    \begin{remark}\label{dim gets big}
        If $\ell(\mathcal{F}_n) \rightarrow \infty$ and the Kahn-Kalai bounds provide nontrivial information, then it is necessary that $\frac{1}{q(\mathcal{F})}\rightarrow \infty$ (as per Remark~\ref{interesting case remark}).  By Lemma~\ref{q estimate}, it is also necessary that $\textrm{dim}\mathcal{F}_n \rightarrow \infty$. Since $\textrm{dim}\mathcal{F}_n \leq \left|(\mathcal{F}_n)_0\right|$, we get the following:
        \begin{corollary}\label{min elem gets big}
            If $\ell(\mathcal{F}_n) \rightarrow \infty$ and the Kahn-Kalai bounds provide nontrivial information, then it is necessary that $\left|(\mathcal{F}_n)_0\right| \rightarrow \infty$.    
        \end{corollary}
    \end{remark}
    %---------------------------------------------%
    More idiosyncratically, we can try to bound the dimension of an upper set using the elementary symmetric polynomials.  Let $\sigma_k$ to denote the $k^{th}$ 'lattice theoretic' elementary symmetric polynomial \cite{ebcba4a9-c880-3a5a-8faa-11dfea58ed2e}, where the addition and multiplication of the classical elementary symmetric polynomial are replaced by the join and meet of the lattice, respectively (e.g., if $S_1,\hdots,S_m \in 2^{X_n}$, then the join is the union and the meet is the intersection, so 
    $\sigma_k(S_1,\hdots,S_m) = \bigcup_{I \subseteq \left\{1,\hdots,m\right\}, \enskip |I|=k}\left(\bigcap_{i \in I}S_i\right)$.)    
    %---------------------------------------------%
    \begin{theorem}\label{dim inequality}
        Suppose $\mathcal{F}$ is a nontrivial upper set in $2^X$ and write $\mathcal{F}_0=\left\{S_1,\hdots,S_{|\mathcal{F}_0|}\right\}$.  For $1 \leq m \leq {|\mathcal{F}_0|}$, if $\textrm{dim}(\mathcal{F}) > |\mathcal{F}_0|-m+1$, then $\sigma_{m}(S_1,\hdots,S_{|\mathcal{F}_0|})=\emptyset$.
    \end{theorem}
    \begin{proof}
        Suppose that $\sigma_m(S_1,\hdots,S_{|\mathcal{F}_0|}) \neq \emptyset$.  Then, there is some subset $I \subseteq \left\{1,\hdots,{|\mathcal{F}_0|}\right\}$ of size $m$ such that $\cap_{i\in I} S_i \neq \emptyset$.  So, the set $\mathcal{G}=\left\{\cap_{i\in I} S_i\right\} \cup \left\{ S_i \in \mathcal{F}_0 : i \not\in I\right\}$ is a cover of $\mathcal{F}$ with $|\mathcal{F}_0|-m+1$ elements and $\textrm{dim}(\mathcal{F})\leq {|\mathcal{F}_0|}-m+1$.  Taking the contrapositive, $\textrm{dim}(\mathcal{F}) > |\mathcal{F}_0|-m+1$ implies $\sigma_m(S_1,\hdots,S_{|\mathcal{F}_0|}) = \emptyset$
    \end{proof} 
    %---------------------------------------------%
    \begin{corollary}\label{dim formula}
        $\textrm{dim}(\mathcal{F}) \leq |\mathcal{F}_0|+1-\textrm{max}\left\{1 \leq m \leq |\mathcal{F}_0|: \sigma_{m}(S_i : S_i \in \mathcal{F}_0)\neq \emptyset\right\}$.  
    \end{corollary}
    %---------------------------------------------%
    \begin{proof}
        Observe that if $\sigma_i(S_1,\hdots,S_{|\mathcal{F}_0|}) \neq \emptyset$, then $\sigma_j(S_1,\hdots,S_{|\mathcal{F}_0|}) \neq \emptyset$ for all $1 \leq j \leq i$ as well.  So,  $t=\textrm{max}\left\{1 \leq m \leq |\mathcal{F}_0|: \sigma_{m}(S_i : S_i \in \mathcal{F}_0)\neq \emptyset\right\}$ is well defined, since $\sigma_1(S_i : S_i \in \mathcal{F}_0) = \bigcup_{S_i \in \mathcal{F}_0}S_i \neq \emptyset$.  Suppose $\textrm{dim}(\mathcal{F}) = n$.  Then, $\textrm{dim}(\mathcal{F})>n-1 = |\mathcal{F}_0| - (|\mathcal{F}_0|-n+2)+1$ and $\sigma_{|\mathcal{F}_0|-n+2}(S_1,\hdots,S_{|\mathcal{F}_0|})=\emptyset$ by Theorem~\ref{dim inequality}.  So, $t < |\mathcal{F}_0|-n+2 = |\mathcal{F}_0|-\textrm{dim}(\mathcal{F})+2$  and $\textrm{dim}(\mathcal{F}) < |\mathcal{F}_0| - t+2$. Thus, $\textrm{dim}(\mathcal{F}) \leq |\mathcal{F}_0| - t +1$. 
    \end{proof}
    %---------------------------------------------%
    Using the Remark~\ref{dim gets big}, this gives us the following direct analog to Theorem~\ref{expectation threshold for prinicpal upper set} and Lemma~\ref{intersection no-go lemma}:
    %---------------------------------------------%
    \begin{corollary}\label{upper bound dim cor}
        Suppose $\ell(\mathcal{F}_n)\rightarrow \infty$ and suppose that the Kahn-Kalai bounds provide nontrivial information.  Then, the intersection of all but any fixed number of the minimal elements of $\mathcal{F}_n$ must eventually be empty--i.e.; for any fixed $t$, it is necessary that there exists a value $N$ such that $\sigma_{|(\mathcal{F}_n)_0|-t}(S_i : S_i \in (\mathcal{F}_n)_0)= \emptyset$ for all $n \geq N$.
    \end{corollary}
    %---------------------------------------------%
    \begin{proof}
        If $\sigma_{|(\mathcal{F}_n)_0|-t}(S_i : S_i \in (\mathcal{F}_n)_0) \neq \emptyset$, then $$\textrm{max}\left\{1 \leq m \leq |(\mathcal{F}_n)_0|: \sigma_{m}(S_i : S_i \in (\mathcal{F}_n)_0)\neq \emptyset\right\} \geq |(\mathcal{F}_n)_0|-t$$ and via Corollary~\ref{upper bound dim cor}, $\textrm{dim}(\mathcal{F}_n) \leq |(\mathcal{F}_n)_0|+1-|(\mathcal{F}_n)_0|+t=t+1$.  If there does not exist a value $N$ such that this does not occur for all $n \geq N$, then it is not possible to have $\textrm{dim}(\mathcal{F}_n) \rightarrow \infty$ as is required (see Remark~\ref{dim gets big}) for the Kahn-Kalai bounds to provide nontrivial information.
    \end{proof}
        %---------------------------------------------%
    Lemma~\ref{intersection no-go lemma} is just Corollary~\ref{upper bound dim cor} with $t=0$ (and the caveat that $\ell(\mathcal{F}_n)\rightarrow \infty$), so this is clearly a much stronger condition.  Restated in the form of Theorem~\ref{expectation threshold for prinicpal upper set} and Lemma~\ref{intersection no-go lemma}, it says that if subsets of the minimal elements \textit{excluding any fixed number of them} do not eventually have empty intersection, the Kahn-Kalai bounds do not provide perfect information.  Essentially, picturing $\mathcal{F}_n$ in the Hasse diagram of $2^{X_n}$, this says that the chunk it occupies must be `wedge-shaped' and never too skinny: the further $\mathcal{F}_n$ climbs up $2^{X_n}$ in one area, the further it must spread down and across $2^{X_n}$ in another. 
%---------------------------------------------%
\section{Conclusion}
%---------------------------------------------%
%---------------------------------------------%
    The Park-Pham theorem (previously known as the Kahn-Kalai conjecture), bounds the critical probability, $p_c(\mathcal{F})$, of a non-trivial subset $\mathcal{F}\subseteq 2^X$ that is closed under supersets.  However, the bounds the result provides--i.e.; $p_c(\mathcal{F})\leq Kq(\mathcal{F})\log\ell(\mathcal{F})$, do not always give any more information than the trivial bound $p_c(\mathcal{F})<1$.  In particular, we have shown that in the interesting case where $\ell(\mathcal{F}_n)\rightarrow \infty$, there is a relatively strong requirement for us to use this tool to extract more information about $p_c(\mathcal{F})$.  Specifically, the minimal elements of $\mathcal{F}_n$ must eventually grow arbitrarily large in number and become quite dispersed:  For all but finitely many $n$, any collection of all but any fixed number of them must have nonempty intersection.
%---------------------------------------------%
\bibliographystyle{amsplain}
\bibliography{references}
%---------------------------------------------%
%---------------------------------------------%
\end{document}